\documentclass[12pt]{amsart}
\usepackage{amssymb}
\newtheorem{Thm}{Theorem}[section]
\newtheorem{Prop}[Thm]{Proposition}
\newtheorem{Cor}[Thm]{Corollary}
\newtheorem{Lem}[Thm]{Lemma}
\theoremstyle{definition}
\newtheorem{Ex}[Thm]{Example}

\newcommand{\A}{\mathcal{A}}
\newcommand{\C}{\mathbb{C}}
\newcommand{\Orb}{\mathbb{O}}
\newcommand{\U}{\mathcal{U}}
\newcommand{\J}{\mathcal{J}}
\newcommand{\g}{\mathfrak{g}}
\newcommand{\lf}{\mathfrak{l}}
\newcommand{\p}{\mathfrak{p}}
\newcommand{\n}{\mathfrak{n}}
\newcommand{\Nilp}{\mathcal{N}}
\newcommand{\gl}{\mathfrak{gl}}
\newcommand{\HC}{\operatorname{HC}}
\newcommand{\Walg}{\mathcal{W}}
\newcommand{\Hom}{\operatorname{Hom}}
\newcommand{\B}{\mathcal{B}}
\newcommand{\Ext}{\operatorname{Ext}}
\newcommand{\Spec}{\operatorname{Spec}}
\newcommand{\q}{\mathfrak{q}}
\newcommand{\tf}{\mathfrak{t}}
\newcommand{\Weyl}{\mathbb{A}}
\newcommand{\slf}{\mathfrak{sl}}
\newcommand{\gr}{\operatorname{gr}}
\newcommand{\F}{\operatorname{F}}
\newcommand{\OCat}{\mathcal{O}}
\numberwithin{equation}{section}
\title{Quantizations of regular functions on nilpotent orbits}
\author{Ivan Losev}
\thanks{Supported by the NSF grant DMS-1161584}
\address{Department
of Mathematics, Northeastern University, Boston MA 02115 USA}
\email{i.loseu@neu.edu}
\begin{document}
\begin{abstract}
We study the quantizations of the algebras of regular functions on nilpotent orbits.
We show that such a quantization always exists and is unique if the orbit
is birationally rigid. Further we show that, for special birationally rigid orbits,
the quantization has integral central character in all cases but four (one orbit
in $E_7$ and three orbits in $E_8$). We use this to complete the computation
of Goldie ranks for primitive ideals with integral central character
for all special nilpotent orbits but one (in $E_8$). Our main ingredient
are results on the geometry of normalizations of the closures of nilpotent orbits
by Fu and Namikawa.
\end{abstract}
\maketitle
\section{Introduction}
\subsection{Nilpotent orbits and their quantizations}
Let $G$ be a connected semisimple algebraic group over $\C$ and let $\g$ be its Lie algebra.
Pick a nilpotent orbit $\Orb\subset \g$. This orbit is a symplectic algebraic variety
with respect to the Kirillov-Kostant form. So the algebra $\C[\Orb]$ of regular functions
on $\Orb$ acquires a Poisson bracket. This algebra is also naturally graded and the Poisson
bracket has degree $-1$. So one can ask about quantizations of $\Orb$, i.e., filtered
algebras $\A$ equipped with an isomorphism $\gr\A\xrightarrow{\sim}\C[\Orb]$ of graded Poisson
algebras.

We are actually interested in quantizations that have some additional structures mirroring
those of $\Orb$. Namely, the group $G$ acts on $\Orb$ and the inclusion $\Orb\hookrightarrow \g$
is a moment map for this action. We want the $G$-action on $\C[\Orb]$ to lift to
a filtration preserving action on $\A$. Further, we want a $G$-equivariant homomorphism
$U(\g)\rightarrow \A$ such that, for any $\xi\in \g$, the endomorphism
$[\xi,\cdot]:\A\rightarrow \A$ coincides with the differential of the $G$-action
(in other words, $\A$ has to be a Dixmier algebra). A motivation to consider quantizations
of $\C[\Orb]$ of this form comes from attempts to extend the orbit method to reductive 
groups, see, e.g., \cite{McG2} for details. 

We establish the existence of such a quantization $\A$ and we investigate the questions of when
$\A$ is unique and when the kernel of the map $U(\g)\rightarrow \A$ has integral central
character. The latter question is of importance for computing Goldie ranks of primitive
ideals with integral central characters in exceptional algebras (the case of classical
Lie algebras was settled in \cite{W_dim}). We elaborate on our results in the next section.

The questions above are closely related to the representation theory of finite W-algebras
introduced by Premet in \cite{Premet1}, see \cite{W_review,Ostrik_review,Wang_review} for reviews.
Each W-algebra is constructed from a pair $(\g,\Orb)$ of a semisimple Lie algebra $\g$
and a nilpotent orbit $\Orb\subset \g$. Dixmier algebras quantizing $\C[\Orb]$ are closely
related to one-dimensional modules over the W-algebra constructed from $(\g,\Orb)$ with
certain additional properties.

\subsection{Main results}\label{SS_Main_res}
Here are three lists of nilpotent orbits in exceptional Lie algebras that are all exceptional in some
further ways explained below in the paper.
\begin{itemize}
\item[(e1)] $\tilde{A}_1$ in $G_2$, $\tilde{A}_2+A_1$ in $F_4$, $(A_3+A_1)'$
in $E_7$, $A_3+A_1, A_5+A_1, D_5(a_1)+A_2$ in $E_8$.
\item[(e2)] $A_4+A_1$ in $E_7$, $A_4+A_1, E_6(a_1)+A_1$ in $E_8$.
\item[(e3)] $A_4+2A_1$ in $E_8$.
\end{itemize}

\begin{Thm}\label{Thm:main}
The following is true.
\begin{enumerate}
\item For any nilpotent orbit $\Orb$, there is a Dixmier algebra $\A$ quantizing
$\C[\Orb]$.
\item If $\Orb$ is birationally rigid (see Section \ref{SS_induct} below for a definition), then $\A$ in (1) is unique.
\item If $\Orb$ is special but not one of the four orbits in (e2),(e3),  then $\A$ from (ii) has integral central character. For orbits (e2) and (e3), $\A$ does not have integral central character.
\end{enumerate}
\end{Thm}

\begin{Cor}\label{Cor:Goldie}
Let $\J$ be a primitive ideal in $U(\g)$ with integral  central character. Suppose that the associated
orbit is not as in (e3). Then the Goldie rank of
$\J$ coincides with the dimension of the corresponding W-algebra module.
\end{Cor}

In \cite{W_dim} we have obtained basically Kazhdan-Lusztig type formulas for the dimensions
of the irreducible finite dimensional W-algebra modules with integral central character. 
So we can view Corollary \ref{Cor:Goldie} as a formula for the Goldie ranks. This corollary has
been already proved for classical types in \cite{W_dim}. 

\subsection{Content of the paper}
We start by recalling various properties of nilpotent orbits: their classification, the notion of a special
orbit, Lusztig-Spaltenstein induction and (birationally) rigid orbits, the structure of the boundary,
and $\mathbb{Q}$-factorial terminalizations.

In Section \ref{S_W_alg} we recall some known results about W-algebras. First, we recall their definition
following the approach taken in \cite{Wquant} and refined in further papers by the author,
see, e.g., \cite{W_prim}. Then we recall an important construction from \cite{HC}: functors between
the categories of Harish-Chandra bimodules. After that we recall the category $\mathcal{O}$
from \cite{BGK} and related constructions. Next, we recall the classification, \cite{LO}, of
finite dimensional irreducible modules with integral central characters over W-algebras.
Finally, we explain some constructions and results related to quantizations of
symplectic varieties, including nilpotent orbits.

In Section \ref{S_proofs} we prove our main results. An important auxiliary result is
Theorem \ref{Thm:equi} that gives a sufficient condition for a functor
$\bullet_{\dagger}$ from \cite{HC} between suitable categories of bimodules to be
an equivalence. We also derive some corollaries, see Section \ref{SS_conseq}.

{\bf Acknowledgements}. I would like to thank George Lusztig, Sasha Premet and David Vogan for stimulating discussions.
This work was partially supported by the NSF under grants DMS-1161584, DMS-1501558.

\section{Preliminaries on nilpotent orbits}
\subsection{Classification and special orbits}
First, let us recall the classification of nilpotent orbits in semisimple Lie algebras.
The nilpotent orbits in $\mathfrak{sl}_n$ are classified by the partitions of $n$.
We sometimes will write partitions as $m^{d_m}(m-1)^{d_{m-1}}\ldots 1^{d_1}$, where
superscripts indicate multiplicities.  The nilpotent orbits in $\mathfrak{so}_{2n+1}$
are classified by the
partitions of  $2n+1$ that have type $B$ meaning that every even part appears even number of times (to an
orbit we assign its Jordan type in the tautological representation of dimension $2n+1$).
The nilpotent orbits in $\mathfrak{sp}_{2n}$ are classified by partitions
of $2n$ that have type $C$ meaning that every odd part appears even number of times.
The nilpotent orbits of $\operatorname{O}_{2n}$ in $\mathfrak{so}_{2n}$  are classified by partitions of
$2n$ of type $D$ meaning that every even part appears even number of times.
The $\operatorname{SO}_{2n}$-action on the $\operatorname{O}_{2n}$-orbit corresponding
to a partition $\mu$ is transitive if and only if there is an odd part in $\mu$,
otherwise  the $\operatorname{O}_{2n}$-orbits splits into the union of
two $\operatorname{SO}_{2n}$-orbits. For a partition $\mu$, we will write
$\Orb_\mu$ for the corresponding orbit (for $\g=\mathfrak{so}_{2n}$
we consider orbits for $\operatorname{O}_{2n}$).

The classification in the exceptional types
is also known, there the orbits have labels consisting of a Dynkin diagram type
(e.g., $A_3+A_2$), sometimes with some additional decoration. The Dynkin diagram is
that of a minimal Levi subalgebra containing an element of a given orbit.

%We have a natural order on the set of orbits, by the inclusion of closures.
%In the classical types, we have $\Orb_\lambda\leqslant \Orb_\mu$ if $\lambda\leqslant \mu$,
%where the latter means $\sum_{i=1}^k \lambda_i\leqslant \sum_{i=1}^k \mu_i$ for all $k$.
%The order in exceptional types is also known, see \cite[Section 13.4]{Carter}.
%
%Let us explain how to compute the dimensions of orbits:
%\begin{itemize}
%\item[$A_{n-1}$:] $\dim \Orb_\mu= n^2-\sum_{i=1}^\infty (\mu^t_k)^2$.
%\item[$B_n$:] $\dim \Orb_\mu=2n^2+n-\frac{1}{2}\sum_{i=1}^\infty (\mu^t_k)^2+\frac{1}{2}\mathsf{odd}(\mu)$,
%where $\mathsf{odd}(\mu)$ denotes  the number of odd parts of $\mu$.
%\item[$C_n$:] $\dim \Orb_\mu=2n^2+n-\frac{1}{2}\sum_{i=1}^\infty (\mu^t_k)^2-\frac{1}{2}\mathsf{odd}(\mu)$.
%\item[$D_n$:]  $\dim \Orb_\mu=2n^2+n-\frac{1}{2}\sum_{i=1}^\infty (\mu^t_k)^2+\frac{1}{2}\mathsf{odd}(\mu)$.
%\end{itemize}
%Dimensions of orbits in exceptional Lie algebras can be found, e.g., in \cite[Section 8.4]{CM}.

Below we will need some information about so called {\it special} orbits (as defined by Lusztig,
\cite[13.1.1]{orange}).
All orbits in type A are special. An orbit in type $B$ or $C$ corresponding to a partition
$\mu$ is special if $\mu^t$ has type $B$ or $C$, respectively. An orbit in type $D$ is special
if $\mu^t$ has type $C$. Special orbits in the exceptional algebras have been classified as well,
see \cite[Section 13.4]{Carter} or \cite[Section 8.4]{CM}.

There is an order reversing bijection ${\bf d}$ (called the Barbasch-Vogan-Spaltenstein duality,
a somewhat implicit construction was earlier discovered by Lusztig)
between the sets of special orbits in $\g$ and in the Langlands dual algebra $\,^L\g$.
For classical types, it is described combinatorially, see \cite[Section 6.3]{CM}, we will not
need this description. The description of this duality in exceptional types is provided
in \cite[Section 13.4]{Carter}.

\subsection{Structure of the boundary}
Below we will need some information about singularities of the closures $\overline{\Orb}$.
We start by studying the situation when $\operatorname{codim}_{\overline{\Orb}}\partial\Orb\geqslant 4$.

The following claim is Proposition 1.3.2 in \cite{Namikawa_induced}.

\begin{Lem}\label{Lem:codim4_class}
Suppose $\g$ is classical.
For $\Orb=\Orb_\mu$, the inequality $\operatorname{codim}_{\overline{\Orb}}\partial\Orb\geqslant 4$
is equivalent to $\mu_i-\mu_{i+1}\leqslant 1$ for all $i$.
\end{Lem}

The following is the list of the orbits $\Orb$ in exceptional Lie algebras that satisfy
$\operatorname{codim}_{\overline{\Orb}}\partial\Orb\geqslant 4$. This information
can be extracted, for example, from \cite[Section 13]{FJLS}.

\begin{itemize}
\item[$G_2$:] $A_1$.
\item[$F_4$:] $A_1,\tilde{A}_1, A_1+\tilde{A}_1, A_2+\tilde{A}_1$.
\item[$E_6$:] $A_1,2A_1,3A_1, A_2+A_1, A_2+2A_1, 2A_2+A_1$.
\item[$E_7$:] $A_1,2A_1, (3A_1)', 4A_1, A_2+A_1, A_2+2A_1, 2A_2+A_1, A_4+A_1$.
\item[$E_8$:] $A_1,2A_1,3A_1,4A_1,A_2,A_2+A_1,A_2+2A_1, A_2+3A_1, 2A_2+A_1,
2A_2+2A_1,A_3+2A_1,D_4(a_1)+A_1, A_3+A_2+A_1, A_4+A_1, 2A_3, A_4+2A_1,
D_4(a_1)+A_1, A_4+A_3$.
\end{itemize}

Let us proceed to explaining results about the structure of the boundary of
the normalization $\operatorname{Spec}(\C[\Orb])$.

\begin{Lem}\label{Lem:smooth_trans}
The codimension of $\operatorname{Spec}(\C[\Orb])^{reg}\setminus \Orb$ in $\operatorname{Spec}(\C[\Orb])$
is $\geqslant 4$ for all but the following   cases:
\begin{itemize}
\item[$G_2$:] $\tilde{A}_1$,
\item[$F_4$:] $\tilde{A}_2+A_1, C_3(a_1)$,
\item[$E_6$:] $A_3+A_1$,
\item[$E_7$:] $(A_3+A_1)', D_6(a_2)$,
\item[$E_8$:] $A_3+A_1, A_5+A_1, D_5(a_1)+A_2, D_6(a_2), E_6(a_3)+A_1, E_7(a_2), E_7(a_5)$.
\end{itemize}
\end{Lem}
Here and below the superscript ``reg'' means the smooth locus.
\begin{proof}

The smooth locus in $\overline{\Orb}$ coincides with $\Orb$. So the inequality
$\operatorname{codim}\operatorname{Spec}(\C[\Orb])^{reg}\setminus \Orb\geqslant 4$
is equivalent to the claim that, for any  orbit $\Orb'$
of codimension $2$ in $\overline{\Orb}$, the normalization
of a formal slice to $\Orb'$ in $\Orb$ is not smooth. By results of \cite{KP1,KP2}, this is
always true when $\g$ is classical. Now we consider exceptional algebras. By results of \cite{FJLS},
our condition is equivalent to at least one edge going down from the label of $\Orb$ in the graphs
of Section 14 is marked with an ``m''. Examining these tables we get a required result.
\end{proof}

\subsection{Induction and birational induction}\label{SS_induct}
Let $\lf$ be a Levi subalgebra in $\g$ and $\Orb'\subset \lf$ be a nilpotent orbit.
Pick a parabolic subalgebra $\p\subset \g$ with Levi subalgebra $\lf$ and let
$\mathfrak{n}$ denote the nilpotent radical of $\mathfrak{p}$. Let $P\subset G$
denote the corresponding parabolic subgroup. Then the fiber bundle $G*_P(\overline{\Orb}'\times\n)$
naturally maps to $\Nilp$, this is known as the {\it generalized
Springer map}. Obviously, there is a unique dense orbit $\Orb$ in the image. This orbit
is called {\it induced from $\Orb'$} (see \cite{LS}), in fact, it is independent of the choice of $\p$. If the map $G*_P(\overline{\Orb}'\times\n)\rightarrow \overline{\Orb}$ is birational, then we say that $\Orb$
is {\it birationally induced} from $\Orb'$.
An orbit that cannot be (birationally) induced from an orbit in a proper Levi is called {\it (birationally) rigid}.

The induction for classical Lie algebras can be described combinatorially on the level of partitions.
We say that a partition $\mu$ of type $X$ (where $X$ is $B,C$ or $D$) is obtained from a
partition $\mu'$ {\it by an elementary step} if
\begin{itemize}
\item[(i)] either there is $n$ such that $\mu_i=\mu'_i+2$ for $i\leqslant n$,
and  $\mu_{i}=\mu_i'$ for $i>n$,
\item[(ii)] or there is $n$  such that  $\mu_i=\mu'_i+2$ for $i<n$,
$\mu_n=\mu'_n+1, \mu_{n+1}=\mu'_{n+1}+1, \mu_i=\mu_i'$ for $i>n+1$,
and the partition obtained from $\mu',n$ in (i) does not have a correct type.
\end{itemize}
The orbit $\Orb_\mu\subset \mathfrak{so}_{|\mu|}$ is induced from the orbit
$\Orb_{|\mu'|}\times \{0\}\subset \mathfrak{so}_{|\mu'|}\times \mathfrak{gl}_?\times\ldots
\mathfrak{gl}_?$ if and only $\mu$ is obtained from $\mu'$ by a sequence
of elementary steps (and the same is true for $\mathfrak{sp}$'s).

\begin{Prop}\label{Prop:bir_ind_rigid}
Any special orbit can be birationally induced from a birationally rigid special orbit.
\end{Prop}
\begin{proof}
{\it Classical types}. We claim that the induction is birational if and only if
only elementary steps of type (i) are involved. The ``if'' part is established
in \cite[5.4.2-5.4.4]{W_dim}. To establish the ``only if'' part, we just need
to show that, for an elementary step of type (ii), the generalized Springer
map is not birational.

Let us consider the orthogonal case, the symplectic one is similar. Note that
in (ii), the number $\mu'_{n}$ and $\mu'_{n+1}$ coincide and are even.
The orbit $\Orb_\mu$ is induced from $\Orb_{\mu'}\times \{0\}\subset \mathfrak{so}_m\times \mathfrak{gl}_n$.
Recall that $P$ acts transitively on $(\Orb'\times \mathfrak{n})\cap \Orb$, see
\cite[Theorem 1.3]{LS}. So if
we know that $Z_G(e)\not\subset Z_P(e)$ for a particular choice of $e$ in
$(\Orb'\times \mathfrak{n})\cap \Orb$, then the corresponding induction
is not birational. Split $\C^m$ into the  sum $(U\oplus U^*)\oplus V'$, where $\dim U=\mu_n'$, the subspaces
$U\oplus U^*$ and $V'$ are orthogonal to one another, and $U$ is lagrangian
in $U\oplus U^*$. We can then split $\C^{m+2n}$ as $(\tilde{U}\oplus \tilde{U}^*)\oplus \tilde{V}$
with $\tilde{U}=U\oplus \C$. We have an element in $(\Orb'\times \mathfrak{n})\cap \Orb$
of the form $e_1+e_2$, where $e_1\in \mathfrak{so}(\tilde{V}')$ and
$e_2\in \gl(\tilde{U})\hookrightarrow \mathfrak{so}(U\oplus U^*)$
is a single Jordan block. Then the summand $\C\subset \tilde{U}$ is
the kernel of $e_2$. But note that the centralizer of $e_2\in \mathfrak{so}(U\oplus U^*)$
does not preserve the kernel of $e_2$ in $\tilde{U}$. This finishes
the argument for $\g$ of orthogonal type.

So all ``weakly rigid'' orbits  listed in \cite[5.4.2-5.4.4]{W_dim} are birationally rigid.
Since any special orbit can be birationally induced from one of those, our claim follows.

{\it Exceptional types}.  We proceed by the induction on $\operatorname{rk}\g$.
We need to show that if a special orbit is not birationally rigid, then it can be
birationally induced from a special orbit.

First, we consider the case when the group $A(\Orb)$ is trivial.
Here any induction is birational, so $\Orb$ is birationally rigid if
and only if it is rigid. By \cite[Theorem 8.3.1]{CM}, if the dual of $\Orb$ intersects
a Levi subalgebra, then $\Orb$ is properly induced from a special orbit.
So we only need to consider orbits whose duals are distinguished (and are not distinguished
themselves, those are obviously induced). The duality is described in \cite[Section 13.4]{Carter}.
Examining the data from there we arrive at the following list of
orbits whose duals are distinguished and that are not distinguished:
\begin{itemize}
\item[($G_2$)] None.
\item[($F_4$)] $\tilde{A}_1,\tilde{A}_1+A_1$.
\item[($E_6$)] $A_1,A_2$.
\item[($E_7$)] $A_1,2A_1,A_2,A_2+2A_1,D_4(a_1)$.
\item[($E_8$)] $A_1,2A_1,A_2, A_2+A_1,A_2+2A_1,2A_2,D_4(a_1),D_4(a_1)+A_1, D_4(a_1)+A_2$.
\end{itemize}
According to tables in \cite[Section 8.4]{CM},  the orbits $\tilde{A}_1$ in $F_4$,
$A_2$ in $E_6$, $A_2, D_4(a_1)$ in $E_7$, $A_2,A_2+A_1,2A_2, D_4(a_1), D_4(a_1)+A_1, D_4(a_1)+A_2$
in $E_8$ have nontrivial groups $A(\Orb)$. The remaining orbits in the list above are rigid,
see, e.g., \cite[Section 4]{dGE}.

Let us proceed to the case of nontrivial $A(\Orb)$. According to \cite[Section 3]{Fu} combined
with the lists of special orbits, see, e.g., \cite[Section 13.4]{Carter} or
\cite[Section 8.4]{CM}, the following special orbits have nontrivial $A(\Orb)$ and are not
birationally induced from a zero orbit.
\begin{itemize}
\item[($G_2$)] None.
\item[($F_4$)] None.
\item[($E_6$)] None.
\item[($E_7$)] $A_2+A_1, A_3+A_2, A_4+A_1, D_5(a_1)$.
\item[($E_8$)] $A_3+A_2, A_4+A_1, A_4+2A_1, D_5(a_1), E_6(a_1)+A_1, E_7(a_3), E_7(a_4)$.
\end{itemize}
The orbits $A_2+A_1, A_4+A_1$ in $E_7$, and $A_4+A_1, A_4+2A_1$ in $E_8$ are birationally
rigid by \cite[Proposition 3.1]{Fu}.The orbit $A_3+A_2$ in $E_7$ is birationally induced
from the orbit $(2^21^6,1^2)$ in $D_5+A_1$ that is special. The orbit $D_5(a_1)$
is birationally induced from the orbit $3^22^21^2$ in $D_6$ which is also special.
These computations were done in  \cite[Section 3.3]{Fu}.

Let us proceed to the remaining 5 orbits in $E_8$ following \cite[Section 3.4]{Fu}.
The orbit $A_3+A_2$ is birationally induced from the special orbit $2^21^{10}$ in $D_7$.
The orbit $D_5(a_1)$ is birationally induced from the special orbit $A_2+A_1$ in $E_7$.
The orbit $E_6(a_1)+A_1$ is birationally induced from the special orbit $A_4+A_1$ in $E_7$.
The orbit $E_7(a_3)$ is birationally induced from the   special orbit $3^22^21^2$ in $D_6$.
Finally, the orbit $E_7(a_4)$ is birationally induced from the special orbit $A_3+A_2$
in $E_7$.
\end{proof}

Note that the proof implies that an orbit $\Orb_\mu$ in a classical Lie algebra is
birationally rigid if and only if $\mu$ satisfies the combinatorial condition of
Lemma \ref{Lem:codim4_class}.
Note also that the orbits $A_2+A_1, A_4+A_1$ in $E_7$ and $A_4+A_1, A_4+2A_1$ in $E_8$ are only
birationally rigid but not rigid orbits in the exceptional Lie algebras, see \cite[Proposition 3.1]{Fu}.
This fact together with the classification of rigid orbits, see, e.g., \cite[Section 4]{dGE},
and Lemma \ref{Lem:smooth_trans} imply the following claim.

\begin{Lem}\label{Lem:smooth_trans_rigid}
The only birationally rigid orbits that fail the condition of
Lemma \ref{Lem:smooth_trans} are the six orbits from (e1).
\end{Lem}

\subsection{$\mathbb{Q}$-factorial terminalizations}
Here we are going to recall the main result of \cite{Namikawa_induced} and \cite{Fu}
and explain some corollaries. This result (proved in \cite{Namikawa_induced}
for the classical types and in \cite{Fu} for the exceptional types) can be stated
as follows.

\begin{Prop}\label{Prop:qfacterm}
Let $\Orb$ be a birationally rigid orbit. Then $\operatorname{Spec}(\C[\Orb])$
has $\mathbb{Q}$-factorial terminal singularities.
\end{Prop}

Now let $\Orb$ be an arbitrary orbit. Suppose that $\Orb$ is birationally induced from
$\Orb'\subset \lf$. Then the  morphism $G*_P (\Spec(\C[\Orb'])\times \mathfrak{n})
\rightarrow \operatorname{Spec}(\C[\Orb])$ (obtained by lifting the generalized
Springer map to the normalizations) is a $\mathbb{Q}$-factorial terminalization.

Also recall the following classical result, see, e.g., \cite[Corollary 2.10]{Fu_review}.

\begin{Prop}\label{Prop:sympl}
For any $\Orb$, the variety $\operatorname{Spec}(\C[\Orb])$ has symplectic singularities
in the sense of Beauville.
\end{Prop}

Let us deduce some corollaries from these results. The following result can be proved
along the lines of the proof of \cite[Lemma 12]{Namikawa_flop}.

\begin{Cor}\label{Cor:vanishing1}
Let $\Orb'\subset \lf$ be a birationally rigid orbit and let $\Orb\subset\g$ be birationally
induced from $\Orb'$. Let $X:=G*_P(\operatorname{Spec}(\C[\Orb'])\times \mathfrak{n})$.
Then $H^i(X^{reg}, \mathcal{O}_{X^{reg}})=0$ for $i=1,2$.
\end{Cor}

The next claim follows from Corollary \ref{Cor:vanishing1} and Lemma \ref{Lem:smooth_trans_rigid}.

\begin{Cor}\label{Cor:orb_vanish}
Let $\Orb$ be a birationally rigid orbit that is not one of the six orbits
listed in Lemma \ref{Lem:smooth_trans_rigid}. Then $H^i(\Orb,\mathcal{O}_{\Orb})=0$
for $i=1,2$.
\end{Cor}

We will need one more result about birationally rigid orbits that is a corollary of
Proposition \ref{Prop:qfacterm} but can also be deduced from the classification.

\begin{Lem}\label{Lem:DR_cohom}
Let $\Orb$ be a birationally rigid orbit. Then $H^2_{DR}(\Orb)=\{0\}$.
\end{Lem}
\begin{proof}
Since $\operatorname{Spec}(\C[\Orb])$ is $\mathbb{Q}$-factorial, we conclude that
the group $\Hom(Z_G(e),\C^\times)$ is finite. So $(\q^*)^Q$
is zero. On the other hand, a standard argument shows that this space is
$H^2_{DR}(\Orb)$.
\end{proof}

\section{Preliminaries on W-algebras and quantizations}\label{S_W_alg}
\subsection{W-algebras}
Let $G$ be a reductive algebraic group, $\g$ its Lie algebra. Pick a nilpotent orbit
$\Orb\subset\g$. Choose an element $e\in \Orb$ and include it into an $\slf_2$-triple
$(e,h,f)$. We write $Q$ for the centralizer of $(e,h,f)$ in $G$.

From the triple $(e,h,f)$ we can produce a filtered associative algebra $\Walg$
equipped with a Hamiltonian $Q$-action.  Namely, consider the universal enveloping
algebra $\U=U(\g)$ with its standard PBW filtration $\U=\bigcup_{i\geqslant 0}\U_{\leqslant i}$.
It will be convenient for us to double the filtration and set $F_i \U:=
\U_{\leqslant [i/2]}$. Form the Rees algebra $\U_\hbar:=\bigoplus_{i}(F_i\U) \hbar^i$.
The quotient $\U_\hbar/(\hbar)$ coincides with $S(\g)=\C[\g^*]$. Identify $\g$
with $\g^*$ by means of the Killing form and let $\chi\in \g^*$ be the image of $\g$.
Consider the completion $\U_\hbar^{\wedge_\chi}$ by the topology induced by
the preimage of the maximal ideal of $\chi$. The space $V:=[\g,f]$ is symplectic
with the form given by $\langle\chi, [\cdot,\cdot]\rangle$. So we can form
the  homogenized Weyl algebra $\Weyl_\hbar$ of $V$, i.e., $\Weyl_\hbar$
is  the Rees algebra of the usual
Weyl algebra $\Weyl(V)$. We consider the  completion $\Weyl_\hbar^{\wedge_0}$ in the topology
induced by the maximal ideal of $0\in V$. Both $\U^{\wedge_\chi}_\hbar$ and
$\Weyl_\hbar^{\wedge_0}$ come equipped with an action of $Q\times \C^\times$.
The action of $Q$ on $\U_\hbar^{\wedge_\chi},\Weyl^{\wedge_0}_\hbar$ is
induced from the natural actions of $Q$ on $\g$ and $V$,
respectively. The group $\C^\times$ acts on $\g^*$ via $t.\alpha:=t^{-2}\gamma(t)\alpha$,
where $\gamma:\C^\times\rightarrow G$ is the one-parameter subgroup associated  to
$h$. It acts on $V$ by $t.v:= \gamma(t)^{-1}v$. Finally, we set $t.\hbar:=t\hbar$,
this defines $\C^\times$-actions on $\U_\hbar^{\wedge_\chi},\Weyl_\hbar^{\wedge_0}$
by topological algebra automorphisms that  commute with the $Q$-actions.

It was checked in \cite{Wquant}, see also \cite{HC,W_prim} that there is a $Q\times\C^\times$-equivariant
$\C[\hbar]$-linear embedding $\Weyl_\hbar^{\wedge_0}\hookrightarrow
\U_\hbar^{\wedge_\chi}$ such that we have the decomposition
\begin{equation}\label{eq:decomp}
\U_\hbar^{\wedge_\chi}\cong \Weyl_\hbar^{\wedge_0}\widehat{\otimes}_{\C[[\hbar]]}\Walg_\hbar',
\end{equation}
where we write $\Walg_\hbar'$ for the centralizer of $\Weyl_\hbar^{\wedge_0}$
in $\U_\hbar^{\wedge_\chi}$. This comes with an action of $Q\times \C^\times$.
Let us write $\Walg_\hbar$ for the $\C^\times$-finite part of $\Walg_\hbar'$,
then $\Walg'_\hbar$ is naturally identified with the completion $\Walg_\hbar^{\wedge_\chi}$.
Set $\Walg:=\Walg_\hbar/(\hbar-1)$. This is a filtered algebra with a Hamiltonian
$Q$-action that does not depend on the choice of the embedding $\Weyl_\hbar^{\wedge_0}
\hookrightarrow \U_\hbar^{\wedge_\chi}$ up to an isomorphism preserving the filtration
and the action. See \cite[Section 2.1]{W_prim}. The associated graded algebra $\gr\Walg$
coincides with $\C[S]$, where $S:=e+\ker\operatorname{ad}f$ is the Slodowy slice.

By a result of Ginzburg, which is a footnote in \cite[Section 5.7]{Premet2},
the natural map $Z(\U)\rightarrow Z(\Walg)$ (where $Z$ stands for the center)
is an isomorphism. So we can speak, for example, about $\Walg$-modules
with integral central character.

\subsection{Functors between Harish-Chandra bimodules}\label{SS_HC_fun}
By a $G$-equivariant Harish-Chandra $\U$-bimodule  (or $(\U,G)$-module) we mean a
finitely generated $\U$-bimodule $\B$ such that the adjoint $\g$-action is locally
finite and integrates to an action of $G$. Similarly, we can introduce the
notion of  a $Q$-equivariant HC $\Walg$-bimodule. We write $\HC^G(\U),\HC^{Q}(\Walg)$
for the categories of equivariant HC bimodules.

In \cite{HC}, we have constructed an exact functor $\bullet_{\dagger}:\HC^G(\U)
\rightarrow \HC^{Q}(\Walg)$. Let us recall the construction of the functor.
Pick a $G$-equivariant HC bimodule $\B$ and equip it with a good filtration
compatible with the filtration $\F_i\U$. So the Rees $\C[\hbar]$-module $\B_\hbar:=
R_\hbar(\B)$ is a $G$-equivariant $\U_\hbar$-bimodule.  Consider the completion
$\B_\hbar^{\wedge_\chi}$ in the $\chi$-adic topology. This is a $Q\times \C^\times$-equivariant
$\U_\hbar^{\wedge_\chi}$-bimodule (the action of $Q$ is Hamiltonian,
while the action of $\C^\times$ is not). As was checked in \cite[Proposition 3.3.1]{HC},
$\B_\hbar^{\wedge_\chi}=\Weyl_\hbar^{\wedge_0}\widehat{\otimes}_{\C[[\hbar]]} \underline{\B}_\hbar'$,
where $\underline{\B}'_\hbar$ is the centralizer of $\Weyl_\hbar^{\wedge_0}$.
So $\underline{\B}'_\hbar$ is a $Q\times \C^\times$-equivariant $\Walg_\hbar^{\wedge_\chi}$-bimodule.
One can show that it coincides with the completion of its  $\C^\times$-finite
part $\underline{\B}_\hbar$. We set $\B_{\dagger}:=\underline{\B}_\hbar/(\hbar-1)$.
This is an object in $\HC^Q(\Walg)$ that comes equipped with a good filtration.
This filtration depends on the choice of a filtration on $\B$, while $\B_{\dagger}$
itself does not.

Let us list properties of the functor $\bullet_{\dagger}$ established in \cite[Sections 3.3,3.4]{HC}.

\begin{Lem}\label{Lem:dag_prop}
The following is true:
\begin{enumerate}
\item $\U_\dagger=\Walg$.
\item $\bullet_{\dagger}$ is an exact functor.
\item $\bullet_{\dagger}$ intertwines the tensor products.
\item $\gr\B_{\dagger}$ (with respect to the filtration above)
coincides with the pull-back of $\gr\B$ to $S$.
\item In particular, $\bullet_{\dagger}$ maps the category $\HC_{\overline{\Orb}}^G(\U)$
of all HC bimodules supported on $\overline{\Orb}$ to the category $\HC^Q_{fin}(\Walg)$
of all finite dimensional $Q$-equivariant $\Walg$-bimodules. Further, $\bullet_{\dagger}$
annihilates $\HC^G_{\partial\Orb}(\U)$.
\item The induced functor $\HC^G_{\Orb}(\U):=\HC_{\overline{\Orb}}^G(\U)/\HC_{\partial\Orb}^G(\U)
\rightarrow \HC^Q_{fin}(\Walg)$ is a full embedding whose image is closed under taking subquotients.
\item The functor $\bullet_{\dagger}$ respects central characters on the left and on the right.
\item For $\B\in \HC_{\overline{\Orb}}(\U)$, the dimension of $\B_{\dagger}$ coincides
with the multiplicity of $\B$ on $\Orb$.
\end{enumerate}
\end{Lem}

The functor $\bullet_{\dagger}: \HC^G_{\overline{\Orb}}(\U)\rightarrow
\HC^Q_{fin}(\Walg)$ has a right adjoint $\bullet^{\dagger}:
\HC^Q_{fin}(\Walg)\rightarrow \HC_{\overline{\Orb}}(\U)$. We will need
the construction of the functor $\bullet^{\dagger}$ below so let us recall it.

Pick $\underline{\B}\in \HC^Q_{fin}(\Walg)$ and equip it with a
$Q$-stable filtration (we can just take the trivial one).
Then form the Rees bimodule $\underline{\B}_\hbar$
and the $Q$-equivariant $\U_\hbar^{\wedge_\chi}$-bimodule
$\B'_\hbar=\Weyl_\hbar^{\wedge_0}\widehat{\otimes}_{\C[[\hbar]]}
\underline{\B}^{\wedge_\chi}_\hbar$. Now set
$$\mathcal{F}_G(\B'_\hbar):=\bigoplus_V V\otimes \operatorname{Hom}_\g(V, \B'_\hbar),$$
where we view $\B'_\hbar$ as a $\g$-module with respect to the adjoint action,
$\xi.b:=\frac{1}{\hbar^2}[\xi,b]$ and the sum is taken over all $G$-modules $V$.
In other words, $\mathcal{F}_G(\B'_\hbar)$ is the maximal subspace in $\B'_\hbar$,
where the adjoint action of $\g$ is locally finite and integrates to an action of
$G$. The space $\mathcal{F}_G(\B'_\hbar)$ is a $G$-equivariant bimodule
over the algebra $\U^{\diamondsuit}_\hbar$ that is the $\g$-finite part of $\U_\hbar^{\wedge_\hbar}$.
Note that the subspace $\mathcal{F}_G(\B'_\hbar)$ is $Q\times \C^\times$-stable in $\B'_\hbar$.

\begin{Lem}\label{Lem:fin_part_fin_gen}
$\mathcal{F}_G(\B'_\hbar)$ is finitely generated as a left $\U^{\diamondsuit}_\hbar$-bimodule.
\end{Lem}
\begin{proof}
We have $\mathcal{F}_G(\B'_\hbar)/\hbar \mathcal{F}_G(\B'_\hbar)\hookrightarrow
\mathcal{F}_G(\B'_\hbar/\hbar \B'_\hbar)$. The latter is a finitely generated
$\C[\g^*]=\U^{\diamondsuit}_\hbar/(\hbar)$-module
by results of \cite[Section 3.3]{HC}. So $\mathcal{F}_G(\B'_\hbar)/\hbar \mathcal{F}_G(\B'_\hbar)$
is a finitely generated $\C[\g^*]$-module as well. Pick $b$ lying in
the $V$-isotypic component of $\B'_\hbar$. The module $\B'_\hbar/\hbar \B'_\hbar$
is generated by finitely many elements $\underline{b}_1,\ldots,\underline{b}_k$
lying, say, in the components $V_1,\ldots,V_k$. Then we can argue by induction on $k$ to
show that, modulo $\hbar^k$, the element $b$ is represented as $\sum_{i=1}^k a^k_i b_i$,
where $b_i$ is a lift of $\underline{b}_i$ to the $V_i$-isotypic component
of $\B'_\hbar$, and $a^k_i$ lies in the sum of the isotypic components of $\U^{\diamondsuit}_\hbar$
corresponding to summands of $V\otimes V_i^*$. Furthermore, we can achieve that
the sequence of $a_i^k$ has limit for $k\rightarrow \infty$. This proves that the elements $b_1,\ldots, b_k$.
\end{proof}

On $\mathcal{F}_G(\mathcal{B}'_\hbar)$ we also have an action of $Q$ restricted from $G$.
The two actions coincide on $Q^\circ$ and their difference is an action of $A:=Q/Q^\circ$
commuting with the $G$-action, see \cite[Section 3.3]{HC}. Let $\B_\hbar$ denote the $\C^\times$-finite part in
$\mathcal{F}_G(\B_\hbar')^{A}$. Then we set $\underline{\B}^{\dagger}:=\B_\hbar/(\hbar-1)$. Again,
$\underline{\B}\mapsto \underline{\B}^{\dagger}$ is a functor and one can show that
it is right adjoint to $\bullet_{\dagger}$. Moreover, the composition of
$\bullet^{\dagger}$ with the forgetful functor $\HC^G_{\overline{\Orb}}(\U)
\twoheadrightarrow \HC^G_{\Orb}(\U)$ is the left inverse of $\bullet_{\dagger}:
\HC_{\Orb}(\U)\rightarrow \HC^Q_{fin}(\Walg)$.

\subsection{Categories $\mathcal{O}$}\label{SS_Cat_O}
One can define the categories $\mathcal{O}$ for $\Walg$, see \cite{BGK}.
Namely, we have the Lie algebra homomorphism $\q\rightarrow \Walg$
that can easily be shown to be an embedding, see \cite[Section 2.1]{W_dim}.
Pick a regular integral element $\theta\in \q$, the integrality refers to a choice of a maximal
torus. We define the category $\OCat^\theta_{\Walg}$ to be the full
subcategory in the category of the finitely generated $\Walg$-modules
consisting of all modules $M$ such that  the real parts of eigenvalues
are bounded from above and all generalized eigenspaces of $\theta$
are finite dimensional. This definition is easily seen to be equivalent
to that in \cite[Section 4.4]{BGK}.

In $\OCat^\theta_{\Walg}$ we have analogs of Verma modules. To define
those we need some notation. The element $\theta$ defines the
grading on $\Walg$, we have $\Walg=\bigoplus_i \Walg_i$, where
$\Walg_i$ is the eigenspace of $[\theta,\cdot]$ with eigenvalue $i$.
We set $\Walg_{\geqslant 0}:=\bigoplus_{i\geqslant 0}\Walg_i$
and $\Walg_{>0}:=\bigoplus_{i>0}\Walg_i$. Note that $\Walg_{>0}$
acts locally nilpotently on a module from  $\mathcal{O}^\theta_{\Walg}$.
The algebra $\Walg^0:=\Walg_{\geqslant 0}/(\Walg_{\geqslant 0}
\cap \Walg\Walg_{>0})$ acts on the annihilator $M^{\Walg_{>0}}$.
One can show that $M^{\Walg_{>0}}$ is a finite dimensional
$\Walg^0$-module, this is an easy consequence of
\cite[Corollary 4.12]{BGK}. We have the Verma module functor
$\Delta^\theta:\Walg^0\operatorname{-mod}_{fg}\rightarrow \OCat^\theta_\Walg$
given by $\Delta^\theta(N):=\Walg/\Walg\Walg_{>0}\otimes_{\Walg^0}N$, it
is right adjoint to $M\mapsto M^{\Walg_{>0}}$.

There is a bijection between the irreducible $\Walg^0$-modules
and the irreducible objects in $\OCat^\theta_{\Walg}$: we send
an irreducible $\Walg^0$-module $N$ to the maximal proper quotient
$L^\theta(N)$.

One can describe the algebra $\Walg^0$, see \cite{BGK}. Namely,
consider the W-algebra $\underline{\Walg}$ for the pair
$(\g_0,e)$, where $\g_0$ is the centralizer of $\theta$ in $\g$.
Then there is a $T$-equivariant filtered algebra isomorphism
$\iota:\Walg^0\xrightarrow{\sim}\underline{\Walg}$, see \cite[Theorem 4.3]{BGK}.
An important feature of this isomorphism is its behavior
on $\tf$. Namely, we have $\iota(\xi)=\xi+\langle\delta,\xi\rangle$,
where $\delta$ is  half the character of the action of
$\tf$ on $V_{>0}$, the sum of the $\theta$-eigenspaces in $V$
with positive eigenvalues.

Below we will only need to know $\delta$ for birationally rigid
orbits up to adding a character of a maximal torus  $T\subset Q$. %The following lemma is elementary.

\begin{Lem}\label{Lem:delta_inegrality}
Assume that $G$ is simply connected and $\Orb$ is birationally rigid.
Suppose that $\g_0$ is a standard Levi subalgebra and that the Lie algebra
$\tf$ of $T$ is the center of $\g_0$.
Then up to adding a character of $T$, $\delta$ coincides with
half the sum of the positive roots $\alpha$
satisfying $\langle\alpha,h\rangle =1$.
\end{Lem}
\begin{proof}
We may assume that $\theta$ is dominant. Note that $V_{>0}\oplus \mathfrak{z}_{\g}(e)_{>0}=\g_{>0}$.
Since $G$ is simply connected, the character of the $T$-action on $\Lambda^{top}\g_{>0}$
is divisible by $2$. Note that $\mathfrak{z}_\g(e)\cong \g(0)\oplus \g(1)$ as a $T$-module,
where we write $\g(i)$ for the eigenspace for $[h,\cdot]$ with eigenvalue $i$.
To establish the claim of the lemma it is enough to check that the character of
$T$ on $\Lambda^{top}\g(0)$ is divisible by $2$. For this, we observe that
$(G(0),G(0))$ is simply connected because $G$ is simply connected. Also
as we have seen in the proof of Lemma \ref{Lem:DR_cohom}, the group $Q$
has no invariants in $\q^*$ or, equivalently, in $\q$. Since $Q\subset G(0)$,
we conclude that $Q^\circ\subset (G(0),G(0))$. Since $(G(0),G(0))$ is simply
connected, the character of a maximal torus of $(G(0),G(0))$ on
$\Lambda^{top}\g(0)$ is indeed divisible by $2$.
\end{proof}

We will need to compute $\delta$ (up to adding an integral weight)
in two cases.

\begin{Ex}\label{Ex:E7}
Consider the orbit $A_2+A_1$ in $E_7$.  We use the notation for (simple)
roots in $E_7$ from the table section of \cite{VO}. For a minimal Levi containing
$e$ we take the standard Levi with simple roots $\alpha_1=\epsilon_1-\epsilon_2,\alpha_2=\epsilon_2-\epsilon_3,
\alpha_6=\epsilon_6-\epsilon_7$. So we can take $h=2\epsilon_1-2\epsilon_3+\epsilon_6-\epsilon_7$.
The positive roots of $E_7$ are of the form $\epsilon_i-\epsilon_j$, where $i<j<8$ or $i=8$
and $\epsilon_i+\epsilon_j+\epsilon_k+\epsilon_8$, where $i,j,k$ are pairwise different numbers
less than $8$. There are twelve roots whose pairing with $h$ equals $1$:

$\epsilon_1-\epsilon_6, \epsilon_8-\epsilon_7, \epsilon_i-\epsilon_7$,
$\epsilon_1+\epsilon_i+\epsilon_7+\epsilon_8$, where $i=2,4,5$, $\epsilon_1+\epsilon_3+\epsilon_6+\epsilon_8,
\epsilon_1+\epsilon_3+\epsilon_6+\epsilon_8, \epsilon_i+\epsilon_j+\epsilon_6+\epsilon_8$, where
$i,j\in \{2,4,5\}$.

The sum of these roots equals $\kappa:=5\epsilon_1+4\epsilon_2+\epsilon_3+4\epsilon_4+4\epsilon_5
+3\epsilon_6-\epsilon_7+8\epsilon_8$.

The intersection of $\tf$ with the coroot lattice has basis
$\epsilon_1+\epsilon_2+\epsilon_3-3\epsilon_8, 3\epsilon_8-\epsilon_5-\epsilon_6-\epsilon_7,
2\epsilon_8-\epsilon_6-\epsilon_7, 4\epsilon_8$. The values of these elements
on $\kappa$ are equal $-14,18, 14, 16$. We deduce that $\frac{1}{2}\kappa$
is integral on $\tf$ (where the lattice is the intersection of $\tf$ with the
coroot lattice).
\end{Ex}

\begin{Ex}\label{Ex:E8}
Now consider the orbit $A_4+2A_1$ in $E_8$. We again use the notation from
\cite{VO}. For a minimal Levi containing $e$, we take the standard Levi subalgebra
with simple roots $\alpha_1=\epsilon_1-\epsilon_2, \alpha_2=\epsilon_2-\epsilon_3,
\alpha_3=\epsilon_3-\epsilon_4, \alpha_4=\epsilon_4-\alpha_5, \alpha_7=\epsilon_7-\epsilon_8,
\alpha_8=\epsilon_6+\epsilon_7+\epsilon_8$. So we can take $h=4\epsilon_1+2\epsilon_2-2\epsilon_4-
4\epsilon_5+\epsilon_6+2\epsilon_7$. The positive roots for $E_8$ are of the form $\epsilon_i-\epsilon_j,
1\leqslant i<j\leqslant 9$, $\epsilon_i+\epsilon_j+\epsilon_k$, where $i,j,k\in \{1,2,\ldots,8\}$
are pairwise distinct, $-\epsilon_i-\epsilon_j-\epsilon_9$, where $1\leqslant i<j<9$.
There are 14 positive roots that pair by $1$
with $h$:

$\epsilon_6-\epsilon_8, \epsilon_6-\epsilon_9, \epsilon_2-\epsilon_6, \epsilon_1+\epsilon_2+\epsilon_5,
\epsilon_1+\epsilon_5+\epsilon_7, \epsilon_1+\epsilon_3+\epsilon_4, \epsilon_1+\epsilon_4+\epsilon_8,
\epsilon_2+\epsilon_3+\epsilon_8, \epsilon_2+\epsilon_4+\epsilon_7, \epsilon_3+\epsilon_7+\epsilon_8,
-\epsilon_3-\epsilon_8-\epsilon_9, -\epsilon_1-\epsilon_5-\epsilon_9, -\epsilon_2-\epsilon_4-\epsilon_9,
-\epsilon_4-\epsilon_7-\epsilon_9$.

The sum of these roots equals $\kappa:=2\epsilon_1+2\epsilon_2+\epsilon_3+\epsilon_7-6\epsilon_9$.
A basis in the intersection of $\tf$ with the coroot lattice is given by
the fundamental weights $\pi_5:=\sum_{i=1}^5 \epsilon_i-5\epsilon_9, \pi_6:=\sum_{i=1}^6
\epsilon_i-3\epsilon_9$. Their pairings with $\kappa$ are $35, 23$. So we see that
$\delta$ is not integral.
\end{Ex}

\subsection{Classification of finite dimensional irreducible representations}\label{SS_W_classif}
Here we will review results from \cite{LO} that concern the classification of finite dimensional
irreducible $W$-algebra modules with integral central character.

Namely, recall that a primitive ideal $\J$ with associated variety $\overline{\Orb}$
exists if and only if $\Orb$ is special. Recall that the set of special nilpotent
orbits is in one-to-one correspondence with two-sided cells in $W$. To a two-sided
cell ${\bf c}$ Lusztig assigned the quotient $\bar{A}$ of the component group $A=Q/Q^\circ$,
see \cite[p. 343]{orange}.
One description of $\bar{A}$ is as follows. Consider the cell $W$-module $[{\bf c}]$
and the Springer $W\times A$-module $\operatorname{Spr}_{\Orb}$. Then $\bar{A}$
coincides with the quotient of $A$ by the kernel of the $A$-action on
$\Hom_{W}([{\bf c}], \operatorname{Spr}_{\Orb})$.

Now consider the primitive ideals with regular integral central character and associated variety
$\Orb$. This set is in bijection with the set of left cells in ${\bf c}$. When
the character is not regular, the set of primitive ideals embeds into the set of
left cells. To a left cell $\sigma\subset {\bf c}$ one can assign a subgroup
$H_\sigma\subset \bar{A}$ that was first introduced in \cite{Lusztig_leading}.
It can be defined as follows, see \cite[Section 6]{LO}. Consider the $A$-module
$\Hom_W([\sigma], \operatorname{Spr}_{\Orb})$, where $[\sigma]$ denotes
the cell module. Then $H_\sigma$ is a unique (up to conjugacy) subgroup
in $\bar{A}$ such that the $\bar{A}$-module $\Hom_W([\sigma], \operatorname{Spr}_{\Orb})$
is induced from the trivial $H_\sigma$-module.

Now take a primitive ideal $\J$ with associated variety $\overline{\Orb}$.
The image $\J_{\dagger}$ is a maximal $Q$-stable ideal of finite codimension
in $\Walg$. The simple finite dimensional modules annihilated by
$\J_{\dagger}$ therefore form an $A$-orbit. It was checked in
\cite{Wquant} that every finite dimensional irreducible $\Walg$-module $N$
is annihilated by  $\J_{\dagger}$ for some $\J$ as above that is forced to have
the same central character as $N$. The main result of \cite{LO} is that
the $A$-orbit over a primitive ideal $\J$ with integral central character
coincides with $\bar{A}/H_\sigma$, where $\sigma$ is the left cell corresponding
to $\J$.

This implies the following corollary.

\begin{Cor}
Let $\Orb$ be a special orbit. Then the following is equivalent.
\begin{itemize}
\item[(a)] $\Walg$ has no $A$-stable finite dimensional irreducible
module with integral central character.
\item[(b)] $\Orb$ is one of the orbits in (e2) of Section \ref{SS_Main_res}.
\end{itemize}
\end{Cor}

\subsection{Quantizations of symplectic varieties}
Here $X$ is a smooth algebraic symplectic variety with form $\omega$. We assume that
$X$ comes equipped with a $\C^\times$-action such that $t.\omega=t\omega$. We are going
to define the notion of a filtered quantization of $X$. By the conical topology on $X$,
we mean a topology, where ``open'' means Zariski open and $\C^\times$-stable.
In particular, the structure sheaf $\mathcal{O}_X$ can be viewed as a sheaf of graded
algebras in the conical topology. By a filtered quantization $\mathcal{D}$ of $X$
we mean a sheaf of filtered algebras together with an isomorphism $\gr\mathcal{D}
\xrightarrow{\sim}\mathcal{O}_X$ of sheaves of graded Poisson algebras such that
the filtration on $\mathcal{D}$ is complete and separated.

Now suppose that a reductive group $G$ acts on $X$ in a Hamiltonian way commuting with
the $\C^\times$-action rescaling the form. Suppose that the $G$-action lifts to
a filtration preserving action on $\mathcal{D}$. We say that the lifted action is Hamiltonian
if there is a $G$-equivariant homomorphism $\Phi:\g\rightarrow \Gamma(\mathcal{D})$ such that
$[\Phi(\xi),\cdot]$ coincides with the derivation induced by the $G$-action, for any $\xi\in \g$.
The map $\Phi$ automatically exists provided $H^1_{DR}(X)=0$.

Consider the case when $X$ is a nilpotent orbit. Here $H^1_{DR}(X)=0$.

\begin{Lem}\label{Lem:quant_nilp}
Let $G$ be a semisimple group and $\Orb$ its nilpotent orbit.
There is a natural bijection between the following two sets.
\begin{enumerate}
\item[(a)] The set of quantizations  of $\Orb$
with Hamiltonian $G$-action.
\item[(b)] The set of primitive ideals $\J\subset \U$ with associated
variety  $\overline{\Orb}$ and the multiplicity of $\U/\J$ on $\Orb $ is $1$.
\end{enumerate}
\end{Lem}
A more general statement (that deals with coverings of nilpotent orbits)
can be found in \cite[Theorem 15]{Moeglin}, \cite{quant_nilp} (note that the
set in (b) is in one-to-one correspondence with the set of 1-dimensional
$A$-stable $\Walg$-modules). We provide an independent  proof in our case.
\begin{proof}
Let $\mathcal{D}$ be a quantization. The map $\Phi$ extends to an algebra
homomorphism $\U\rightarrow \Gamma(\mathcal{D})$. We have $\gr\Gamma(\mathcal{D})
\hookrightarrow \C[\Orb]$. Moreover, the composition of $\gr\Phi:\g\rightarrow
\gr\Gamma(\mathcal{D})$ with this inclusion coincides with the comoment map
$\mu^*:\g\rightarrow \C[\Orb]$. So the composition of $\gr\Phi: S(\g)\rightarrow \gr\Gamma(\mathcal{D})$
with the inclusion $\gr\Gamma(\mathcal{D})\hookrightarrow \C[\Orb]$ coincides
with the map $\mu^*:S(\g)\rightarrow \C[\Orb]$.
%It follows that
Set $\J_{\mathcal{D}}:=\ker\Phi$. We have $\gr\J_{\mathcal{D}}\subset \ker\mu^*$,
the latter coincides with the ideal of $\overline{\Orb}$.
In particular, $\overline{\Orb}$ is contained in the associated variety of $\J_{\mathcal{D}}$. On the other hand,
we see that
\begin{equation}\label{eq:inclusions}\operatorname{im}\mu^*\subset \gr(\U/\mathcal{J}_D)
\subset \gr\Gamma(\mathcal{D})\subset \C[\Orb],\end{equation}
where the filtration on $\U/\mathcal{J}_D$ is induced from $\Gamma(\mathcal{D})$.
The algebra $\C[\Orb]$ is finite over $\operatorname{im}\mu^*=\C[\overline{\Orb}]$.
It follows that $\gr\Gamma(\mathcal{D})$ is finitely generated and the GK-dimension
equals $\dim \Orb$. So $\Gamma(\mathcal{D})$ is a HC bimodule.
From the inclusion $\U/\J_{\mathcal{D}}\hookrightarrow \Gamma(\mathcal{D})$
we deduce that the associated variety of $\J_{\mathcal{D}}$
is exactly $\overline{\Orb}$. Also from here we see that the multiplicity is $1$.
Further, the algebra $\Gamma(\mathcal{D})$ has no zero divisors. It follows
that $\mathcal{J}_D$ is a completely prime ideal. It follows that it is primitive.

%has properties described in (b) (the ideal
%%is primitive because it is completely prime).

Given $\J$ as in (b), we can get a quantization  $\mathcal{D}$ as in (a)
by microlocalizing $\U/\J$ to the open subvariety $\Orb\subset\operatorname{Spec}(\gr(\U/\J))$. Let us denote the resulting
quantization by $\mathcal{D}_{\J}$.

Now let us check that the maps $\J\mapsto \mathcal{D}_\J$ and $\mathcal{D}\rightarrow
\mathcal{J}_{\mathcal{D}}$ are dual to one another. First of all,
$\Phi$ factors through $\U/\J_{\mathcal{D}}\rightarrow \Gamma(\mathcal{D})$.
This gives rise to a   homomorphism
$\mathcal{D}_{\J_{\mathcal{D}}}\rightarrow \mathcal{D}$
of filtered sheaves of algebras. From (\ref{eq:inclusions}),
we see that on the level of associated graded sheaves, this
homomorphism is the identity. So  $\mathcal{D}_{\J_{\mathcal{D}}}\xrightarrow{\sim} \mathcal{D}$.

On the other hand, we see that $\J\subset \J_{\mathcal{D}_\J}$. Since both ideals are prime
and have the same associated variety, the equality $\J=\J_{\mathcal{D}_\J}$ follows
now from \cite[Corollar 3.6]{BK}.
\end{proof}

Now let us assume that $H^1(X,\mathcal{O}_X)=H^2(X,\mathcal{O}_X)=0$. In this case there is
classification of quantizations, see \cite{BK_quant,quant_iso}.  Namely, there is a natural
bijection between the set of isomorphism  classes of filtered quantizations of $X$
and $H^2_{DR}(X)$. If a reductive group $G$ is connected, then the $G$-action lifts
to any filtered quantization, see the proof of \cite[Proposition 6.2]{BK}.

\section{Proofs of the main results}\label{S_proofs}
\subsection{$Q$-equivariance}\label{SS_Q_equiv}
Here we will prove two results that give sufficient conditions for a $\q$-action on
a finite dimensional $\Walg$-module $V$ to integrate to an action of
$Q^\circ$.

\begin{Lem}\label{Lem:1dim_equivar}
Let $V$ be a $1$-dimensional $Q$-stable $\Walg$-module. Assume $\Orb$
is birationally rigid. Then the action of $\q$ on $V$ is trivial.
\end{Lem}
So the action of $\q$ on $V$ obviously integrates to an action of $Q^\circ$.
\begin{proof}
The proof of Lemma \ref{Lem:DR_cohom} shows that $(\q^*)^Q=\{0\}$. Let $\chi$
be a character of the $\q$-action on $V$. Since $V$ is $Q$-stable, we see
that $\chi=q.\chi$ for any $q\in Q$ so $\chi$ is $Q$-stable. It follows that
$\chi$ is zero.
\end{proof}

\begin{Prop}\label{Prop:integr_regular}
Let $G$ be simply connected. Let $V$ be an irreducible $\Walg$-module
with integral central character. Then the $\q$-action on $V$
integrates to a $Q^\circ$-action if and only if the character $\delta$
recalled in Section \ref{SS_Cat_O} is integral.
\end{Prop}
\begin{proof}
Let us pick a maximal torus $T\subset Q^\circ$ and let $\tf\subset \q$ be its Lie algebra.
What we need to show is that $\tf$ acts diagonalizably and with characters lying in
the character lattice of $T$. Pick a regular integral element $\theta\in \tf$. Let
$\g^0,\Walg^0$ be as in Section \ref{SS_Cat_O}. Let us write $G_0$ for the connected subgroup of $G$
corresponding to $\g_0$. Note that $T=Z(G_0)^\circ$. Since the group $G$
is simply connected, we see that the character group
of $T$  is spanned by the fundamental weights that vanish on the roots for $\g_0$.

We have $V=L^\theta_{\Walg}(V^0)$, where $V^0$ is an
irreducible $\Walg^0$-module. Our claim reduces to checking that the character of the
$\tf$-action on $V^0$ is integral.

Now recall the isomorphism $\Walg^0\xrightarrow{\sim} \underline{\Walg}$ that induces
the shift by $\delta$ on $\tf$. Let $\underline{V}$ be the $\underline{\Walg}$-module
corresponding to $V^0$. According to \cite[Corollary 4.8]{BGK}, this module has integral central character.
It follows that $\tf$ acts on $\underline{V}$ with integral character, i.e., it action
integrates to $T$. So the action of $\tf$ on $V^0$ integrates to $T$ if and only if
$\delta$ is integral.
\end{proof}

\begin{Cor}\label{Cor:A_stab}
Let $\Orb$ be  birationally rigid and special.  Then the following are equivalent.
\begin{enumerate}
\item $\Walg$ has an $A$-stable 1-dimensional module with integral central character.
\item $\delta$ is integral.
\end{enumerate}
\end{Cor}

In particular, we see that $\delta$ is integral when $\g$ is classical, \cite[Section 5.4]{W_dim},
or when $\Orb$ is, in addition, rigid, \cite[Theorem B]{Premet3}. We remark that checking
the integrality of $\delta$ involves only an elementary combinatorics and can be done in all these
cases directly.

\subsection{Equivalence theorem}
\begin{Thm}\label{Thm:equi}
Suppose that $\operatorname{codim}_{\overline{\Orb}}\partial\Orb\geqslant 4$.
Then the functor $\bullet_{\dagger}:\HC^G_{\Orb}(\U)\rightarrow \HC^Q_{fin}(\Walg)$
is a category equivalence.
\end{Thm}

We will need to introduce some auxiliary categories related to categories of HC bimodules.
Recall that we write $\U^{\diamondsuit}_\hbar$ for the $\g$-finite part of $\U_\hbar^{\wedge_\hbar}$.
We write
$\operatorname{HC}^G(\U^{\diamondsuit}_\hbar)$ for the category of $G$-equivariant
finitely generated $\U^{\diamondsuit}_\hbar$-bimodules.
Note that any such module $\B_\hbar$ becomes a
graded $\U_\hbar$-bimodule by the formula $b\xi=\xi b-\hbar^2 \xi.b$, where $b\in \B_\hbar,
\xi\in \g$ and $\xi.b$
is the image of $b$ under the derivation coming from the $G$-action.
We write $\widetilde{\operatorname{HC}}^G(\U^{\diamondsuit}_\hbar)$
for the ind completion of $\operatorname{HC}^G(\U^{\diamondsuit}_\hbar)$,
it consists of all $G$-equivariant $\U_\hbar^{\diamondsuit}$-modules.
Next, we consider the category $\HC(\U_\hbar^{\wedge_\chi})$ consisting of
the $\g$-equivariant finitely generated $\U_\hbar^{\wedge_\chi}$-modules.
We also consider the subcategories $\HC^G_{\overline{\Orb}}(\U_\hbar^{\diamondsuit})
\subset \HC^G(\U_\hbar^{\diamondsuit})$ of all modules supported on $\overline{\Orb}$
and $\HC_{\Orb^{\wedge_\chi}}(\U_\hbar^{\wedge_\chi})\subset \HC(\U_\hbar^{\wedge_\chi})$
of all modules supported on $\Orb^{\wedge_\chi}$.

Considering the functor $\mathcal{F}=\mathcal{F}_G: \HC_{\Orb^{\wedge_\chi}}(\U_\hbar^{\wedge_\chi})\rightarrow
\HC^G_{\overline{\Orb}}(\U^{\diamondsuit}_\hbar)$ be the functor of taking $G$-finite sections extending
the functor considered in Section \ref{SS_HC_fun}, it is given by $$\mathcal{B}_\hbar'\mapsto
\bigoplus_V V\otimes \Hom_{U(\g)}(V,\mathcal{B}_\hbar'),$$ where the sum is
taken over all irreducible finite dimensional $G$-modules $V$. So the functor $\mathcal{F}$ admits
derived functors $$R^i\mathcal{F}: \mathcal{B}_\hbar'\mapsto
\bigoplus_V V\otimes \operatorname{Ext}^i_{U(\g)}(V,\mathcal{B}_\hbar'):
\HC_{\Orb^{\wedge_\chi}}(\U_\hbar^{\wedge_\chi})\rightarrow
\widetilde{\HC}^G(\U_\hbar^{\diamondsuit}).$$

Here is a main technical result that is needed to establish Theorem \ref{Thm:equi}.

\begin{Lem}\label{Lem:fin_gen}
Let $\B'_\hbar\in\HC_{\Orb^{\wedge_\chi}}(\U_\hbar^{\wedge_\chi})$
and $B:=\B'_\hbar/\hbar \B'_\hbar$.
Then the following is true
\begin{enumerate}
\item all $\g$-isotypic components in $R^1\mathcal{F}(\B'_\hbar)$ are finitely
generated $\C[[\hbar]]$-modules.
\item $R^1\mathcal{F}(\B'_\hbar)$ is a finitely generated
left $\U^{\diamondsuit}_\hbar$-module and $R^1\mathcal{F}(\B'_\hbar)/\hbar R^1\mathcal{F}(\B'_\hbar)$
is supported on $\partial\Orb$.
\item The cokernel of the natural homomorphism
$\mathcal{F}(\B'_\hbar)/\hbar \mathcal{F}(\B'_\hbar)\hookrightarrow
\mathcal{F}(B)$ is supported on $\partial\Orb$.
\end{enumerate}
\end{Lem}
\begin{proof}
The proof is in several steps. Let us write $\tilde{G}$ for the simply
connected cover of $G$.

{\it Step 1}. We start by computing $R^i\mathcal{F}(M)$ for certain
$N\in \HC_{\overline{\Orb}^{\wedge_\chi}}(\U_\hbar^{\wedge_\chi})$. We will consider
the objects annihilated by both $\hbar$ and the ideal of $\Orb$. The category of such objects
was shown in \cite[Section 3.2]{HC} to be equivalent to the category of $Z_{\tilde{G}}(e)^\circ$-modules:
an object $N\in Z_{\tilde{G}}(e)\operatorname{-mod}$ gets sent to the sections $\mathcal{V}^{\wedge_\chi}_N$
of the vector  bundle $\mathcal{V}_N=\tilde{G}*_{Z_{\tilde{G}}(e)^\circ}N$ on the formal neighborhood of $Z_{\tilde{G}}(e)^{\circ}$ in $\tilde{G}/Z_{\tilde{G}}(e)^\circ$. Moreover, it was computed in \cite[Section 3.2]{HC} that
$\mathcal{F}(\mathcal{V}^{\wedge_\chi}_N)=\Gamma(\mathcal{V}_N)$. By the uniqueness of
classical derived functors (where the source category is that of rational $Z_{\tilde{G}}(e)^\circ$-modules),
we see that $R^i\mathcal{F}(\mathcal{V}^{\wedge_\chi}_N)=H^i(\mathcal{V}_N)$
for all $i$ (all global sections and cohomology are taken on $\tilde{\Orb}:=\tilde{G}/Z_{\tilde{G}}(e)^\circ$).

{\it Step 2}. Let us prove that $H^1(\mathcal{V}_N)$ is a finitely generated $\C[\tilde{\Orb}]$-module
supported on $\partial \tilde{\Orb}:=X\setminus \tilde{\Orb}$, where we write
$X$ for $\operatorname{Spec}(\C[\tilde{\Orb}])$. Set $V_N:=\Gamma(\mathcal{V}_N)$,
this is a coherent sheaf on $X$ whose restriction to $\tilde{\Orb}$ coincides
with $\mathcal{V}_N$, see \cite[Section 3.2]{HC}.  Then $H^i(\mathcal{V}_N)=H^{i+1}_{\partial\tilde{\Orb}}(V_N)$.
By \cite[Expose VIII, Cor. 2.3]{Groth}, $H^{i+1}_{\partial\tilde{\Orb}}(V_N)$ is
finitely generated (and obviously supported on $\partial \tilde{\Orb}$) provided
$i+1<\operatorname{codim}_{X}\partial\Orb$. The right hand side in this inequality
is at least 4 by the assumptions of the proposition. The claim in the beginning
of this step follows.

{\it Step 3}. Since any object $B$ in $\HC(\U^{\wedge_\chi}_\hbar)$  annihilated
by $\hbar$ admits a finite filtration by objects of the form $\mathcal{V}^{\wedge_\chi}_N$,
we deduce from Step 2 that for any such $B$, the $\C[\g^*]$-module $R\mathcal{F}^1(B)$
is finitely generated and is supported on $\partial\Orb$.

{\it Step 4}. In this step we prove (1). Set $B:=\B'_\hbar/\hbar \B'_\hbar$. Note that, for any finite
dimensional $\g$-module $V$, we have $\dim\operatorname{Ext}^1_{U(\g)}(V,B)<\infty$.
Indeed, the latter is the $V$-isotypic component in the finitely generated module $R^1\mathcal{F}(B)$
over $A:=\C[\g^*]/\operatorname{Ann}(B)$.  The reduced spectrum of $A$ is $\overline{\Orb}$,
therefore $A^G$ is finite dimensional. The inequality  $\dim\operatorname{Ext}^1_{U(\g)}(V,B)<\infty$
follows from the observation that any $\g$-isotypic component in a finitely generated
$A$-module is a finitely generated $A^G$-module.

Let us complete the proof of (1). The space
$\operatorname{Hom}_{\g}(V,R^1\mathcal{F}(\mathcal{B}_\hbar'))$ coincides with
$\Ext^1_{U(\g)}(V, \mathcal{B}_\hbar')$. The Ext groups
$\Ext^1_{U(\g)}(V, \mathcal{B}_\hbar')$ are computed using the Chevalley-Eilenberg
complex. In this complex, all cochains are complete and separated over $\C[[\hbar]]$.
Arguing as in the end of the proof of \cite[Lemma 5.6.3]{GL}, we conclude that
$\Ext^1_{U(\g)}(V, \mathcal{B}_\hbar')$ is finitely generated over $\C[[\hbar]]$.
%Since $\Ext^1_{U(\g)}(V, \mathcal{B}_\hbar')/ \hbar \Ext^1_{U(\g)}(V, \mathcal{B}_\hbar')
%\hookrightarrow \operatorname{Ext}^1_{U(\g)}(V,B)$, our claim follows.

{\it Step 5}. Let us prove (2). We have the following exact sequence
$$0\rightarrow \mathcal{F}(\mathcal{B}'_\hbar)/\hbar \mathcal{F}(\mathcal{B}'_\hbar)
\rightarrow \mathcal{F}(B)\rightarrow R^1\mathcal{F}(\mathcal{B}_\hbar')
\xrightarrow{\hbar} R^1\mathcal{F}(\mathcal{B}_\hbar')\rightarrow
R^1\mathcal{F}(B).$$
It follows from Step 3 that $R^1\mathcal{F}(B)$
is a finitely generated $\C[\g^*]$-module that is supported on $\partial\Orb$
By Step 4, $R^1\mathcal{F}(\mathcal{B}_\hbar')$ is the direct sum of finitely generated
$\C[[\hbar]]$-modules. The same argument as in the proof of Lemma \ref{Lem:fin_part_fin_gen}
shows that $R^1\mathcal{F}(\mathcal{B}_\hbar')$ is a finitely generated $\U_\hbar^{\diamondsuit}$-module.
Since $R^1\mathcal{F}(B)$ is supported on $\partial \Orb$, then so is
$R^1\mathcal{F}(\mathcal{B}_\hbar')$.

{\it Step 6}. Let us prove (3).
The algebra $\U^{\diamondsuit}_\hbar$ is Noetherian, this is established similarly
to Lemma \ref{Lem:fin_part_fin_gen}. It follows that the kernel of $\hbar$ in
$R^1\mathcal{F}(\mathcal{B}_\hbar')$ admits a finite filtration whose quotients
are subquotients in $R^1\mathcal{F}(\mathcal{B}_\hbar')/\hbar R^1\mathcal{F}(\mathcal{B}_\hbar')$.
Therefore the kernel of $\hbar$ is finitely generated over $\C[\g^*]$ and is supported
on $\partial \Orb$.
\end{proof}

\begin{proof}[Proof of Theorem \ref{Thm:equi}]
We can modify the categories involved in Lemma \ref{Lem:fin_gen} considering weakly
$\C^\times$-equivariant modules for $\U_\hbar^{\diamondsuit}$ and $Q$-equivariant
and Kazhdan weakly $\C^\times$-equivariant bimodules for $\U_\hbar^{\wedge_\chi}$.
Then we have the functor $\mathcal{F}'_G(\bullet)=\mathcal{F}(\bullet)^{A}_{\C^\times\operatorname{-fin}}$
and (2) and (3) of that lemma still hold (with the same proof or as formal corollaries
of those claims). Now pick $\B\in \HC^Q_{fin}(\Walg)$ and apply these claims
to $\B'_\hbar:=\Weyl_\hbar^{\wedge_0}\widehat{\otimes}_{\C[[\hbar]]}R_\hbar(\underline{\B}_\hbar)^{\wedge_\chi}$.
From the construction of $\bullet_{\dagger},\bullet^{\dagger}$ in Section \ref{SS_HC_fun},
we see that (2) and (3) yield $\operatorname{mult}_{\Orb}\mathcal{B}^{\dagger}=\dim \mathcal{B}$.
It follows that $\bullet^{\dagger}$ is exact and faithful when viewed
as a functor $\HC_{fin}^Q(\Walg)\rightarrow \HC_{\Orb}(\U)$. Also it follows
that $\dim (\B^{\dagger})_{\dagger}=\dim \B$.
Since $(\bullet_{\dagger},\bullet^{\dagger})$ is an adjoint pair, our claim
follows.
\end{proof}

\subsection{Consequences of the equivalence theorem}\label{SS_conseq}
\begin{proof}[Proof of Theorem \ref{Thm:main}]
{\it Proof of (1)}. Let $X$ be a $\mathbb{Q}$-factorial terminalization of
$\operatorname{Spec}(\C[\Orb])$. Let $\mathcal{D}$ be a quantization
of $X^{reg}$. The group $\pi_1(X^{reg})$ is finite because
it coincides with  $\pi_1(\Orb)$. It follows that $H^1_{DR}(X^{reg})=0$
and so the $G$-action on $\mathcal{D}$ is Hamiltonian. From $H^1(X^{reg},\mathcal{O}_{X^{reg}})=0$
it follows that $\gr\Gamma(\mathcal{D})=\C[X^{reg}]=\C[\Orb]$.
Since the $G$-action on $\Gamma(\mathcal{D})$ is Hamiltonian, we see that
$\Gamma(\mathcal{D})$ is a Dixmier algebra.

{\it Proof of (2)}. $X^{reg}$ has a unique
quantization because $H^2_{DR}(X^{reg})=\{0\}$. This finishes the proof.

{\it Proof of (3)}. As we have seen in Example \ref{Ex:E7} and Section \ref{SS_Q_equiv},
in all  cases except (e2) and (e3), $\delta$ is integral. Pick a unique $A$-stable 1-dimensional representation
$V$ of $\Walg$ and a finite dimensional representation $U$ of
$\Walg$ with integral central  character. By Lemma \ref{Lem:1dim_equivar} $V$ is $Q^\circ$-equivariant.
By Proposition \ref{Prop:integr_regular}, $U$ is
$Q^\circ$-equivariant. So $\operatorname{Hom}_{\C}(U,V)$ is a
$Q^\circ$-equivariant $\Walg$-bimodule. Set $\mathcal{B}:=Q*_{Q^\circ}
\operatorname{Hom}_{\C}(U,V)$. It is a $Q$-equivariant $\Walg$-bimodule.
Consider $\mathcal{B}^\dagger$. By Theorem \ref{Thm:equi},
this is a nonzero HC $\U$-bimodule. The central character on the right
is integral and therefore so is integral character on the left.

For the three orbits in (e2), there are no $A$-stable irreducible
finite dimensional representations with integral central characters
by Corollary \ref{Cor:A_stab}.
For the orbit in (e3), we have checked in Example \ref{Ex:E8} that $\delta$ is non-integral.
\end{proof}

\begin{Lem}\label{Cor:1dim_unique}
Let $\Orb$ be a birationally rigid orbit but not one of the orbits in (e1).
Then there is a unique primitive ideal $\J\subset \U$ with associated variety
$\overline{\Orb}$ and multiplicity of $\U/\J$ equal to $1$.
\end{Lem}
This lemma also follows from \cite[Theorem 4]{PT}.
Premet has checked in  \cite{Premet3} that for the six orbits listed in (e1) there are
exactly two ideals $\J$ with specified properties.
\begin{proof}
By Corollary \ref{Cor:orb_vanish}, $H^i(\Orb, \mathcal{O}_{\Orb})=0$ for $i=1,2$.
Besides, $H^1_{DR}(\Orb)$ as for any other orbit, and $H^2_{DR}(\Orb)=0$
by Lemma \ref{Lem:DR_cohom}. So there is a unique Hamiltonian quantization of
$\Orb$. The claim of this lemma now follows from Lemma \ref{Lem:quant_nilp}.
%Set $X:=\operatorname{Spec}(\C[\Orb])$.By Lemma \ref{Lem:smooth_trans_rigid}, $\operatorname{codim}_{X}X^{reg}\setminus %\Orb\geqslant 4$. By \cite[Proposition 3.4]{BPW}, the restriction to $\Orb$ defines a bijection
%between the isomorphism classes of quantizations of $\C[\Orb]$ and of $\Orb$. Now we can use
%(2) of Theorem \ref{Thm:main} and Lemma \ref{Lem:quant_nilp} to finish the proof.
\end{proof}

There is a conjectural recipe of how to compute a highest weight for the primitive ideal
$\J$ (at least, under the additional assumption that $\Orb$ is special). 
Namely, let $\Orb^\vee$ be the dual orbit  of $\Orb$ and let $h^\vee$ be a dominant
representative of the semisimple element for the $\mathfrak{sl}_2$-triple for $\Orb^\vee$.
Following \cite{BV}, set $\lambda_{\Orb}:=\frac{1}{2}h^\vee$. Then an expectation is 
that $\lambda_{\Orb}$  is a highest weight for $\J$. This is true for classical
types and special birationally rigid orbits, \cite[Corollary 5.19]{McG2}, and for 
rigid special orbits in exceptional types, \cite[Theorem B]{Premet3}.
Note that for the orbit $A_2+A_1$ in $E_7$, the dual orbit is $E_6(a_1)$,
which is even. So $\lambda_{\Orb}$ in this case is integral. We also remark that \cite[Theorem III]{BV} gives a character
formula (i.e., computes the multiplicities 
of the irreducibles finite dimensional $\g$-modules) 
for $\U/I(\lambda_{\Orb})$, where $I(\lambda_{\Orb})$ denotes the annihilator
of the irreducible module with highest weight $\lambda_{\Orb}$, which holds even
if $\lambda_{\Orb}$ is non-integral. The multiplicities for $\C[\Orb]$
are also known thanks to the work of McGovern, \cite{McG1}. So one could try
to compare the two sets of multiplicities to check if $\U/I(\lambda_{\Orb})$
is multiplicity free. However, this is far from being straightforward as the formulas
are different.

\begin{proof}[Proof of Corollary \ref{Cor:Goldie}]
If $\Orb$ is not one of the orbits in (e2),(e3), then there is a primitive ideal $\J$ with integral
central character and associated variety $\overline{\Orb}$ such that the multiplicity of
$\U/\J$ on $\Orb$ is $1$. It follows that $\Walg$ has an $A$-stable 1-dimensional module.
The claim of the present corollary follows from \cite[Theorem 1.3]{W_dim}.

Consider the three orbits in (e2). Let $\Orb$ be induced from an orbit $\Orb'$ in a Levi
and let $\Walg'$ denote the W-algebra corresponding to $\Orb'$. Recall the exact dimension
preserving functor $\rho:\Walg'\operatorname{-mod}_{fd}\rightarrow \Walg\operatorname{-mod}$,
see \cite[Theorem 1.2.1]{miura}. It maps modules with integral central character
to modules with integral central character by \cite[Corollary 6.4.2]{miura}.
The four orbits in (e2),(e3) are induced from special orbits, this can be shown
using arguments in \cite[Sections 3.3, 3.4]{Fu}. So for all four orbits there
are ($A$-unstable) $\Walg$-modules with integral central characters.
For the orbits in (e2) we are now done by \cite[Theorem 1.3]{W_dim}.
%The orbit $A_4+A_1$ in $E_7$ is induced from the zero orbit
%in a suitable Levi, see, e.g., \cite[Section 3.3]{Fu}, and so has a one-dimensional representation
%with integral central character, this follows from \cite[Corollary 6.4.2]{miura}.  The orbit
%$E_6(a_1)+A_1$ in $E_8$ is induced from $A_4+A_1$ in $E_7$ and so $\Walg$ again has
%a one-dimensional representation with integral central character. The orbit $A_4+A_1$
%in $E_8$ is induced from the special orbit
\end{proof}

So there is just one orbit, where a precise relationship between the dimensions
of irreducible $\Walg$-orbits and the Goldie ranks of primitive ideals (with integral
central character) is still unknown. Here $\bar{A}=\mathbb{Z}/2\mathbb{Z}$.
If the left cell $\sigma$ corresponding  to a primitive ideal $\J$ satisfies
$H_{\sigma}=\{1\}$, then the dimension coincides with the Goldie rank, this
follows from the proof of \cite[Theorem 1.2]{W_dim}. If $H_\sigma=\bar{A}$,
then the ratio of the dimension by the Goldie rank is either $1$  or $2$.

\begin{Prop}\label{Prop:Lusztig_quot}
Let $\Orb$ be a birationally rigid special orbit. Then $A\cong \bar{A}$.
\end{Prop}
This proposition is basically a special case of the description of $\bar{A}$
given in \cite[p. 343]{orange}. An advantage of our approach is that it 
is conceptual. 
\begin{proof}
We use the notation from Section \ref{SS_W_classif}.
Fix a left cell $\sigma$. The number of irreducible HC bimodules whose left
and right annihilators are the primitive ideal (with central character
$\rho$) corresponding to $\sigma$ is equal to  $\sigma\cap \sigma^{-1}=|\operatorname{Irr}(\bar{A})|$.
Applying $\bullet_{\dagger}$ to these irreducible HC bimodules,
we get sheaves on $\bar{A}/H_\sigma$ whose fibers
at $H_\sigma$ are modules over $H_\sigma$. But any representation
$V$ of $A$ defines an endo-functor $\HC^Q_{fin}(\Walg)\rightarrow
\HC^Q_{fin}(\Walg)$ given by $V\otimes\bullet$. If $A\twoheadrightarrow
\bar{A}$ has a kernel, then the image of $\bullet_{\dagger}$ is not stable
under all functors $V\otimes \bullet$.
\end{proof}

\end{document}